\documentclass[11pt]{amsart}
\usepackage{amssymb, amsmath, amsthm}
\usepackage[latin1]{inputenc}
\usepackage{graphicx}
\usepackage[final]{hyperref}
\usepackage{color}
\usepackage{ulem}
\usepackage{epstopdf}

\usepackage[a4paper, centering]{geometry}
\usepackage{color}
\usepackage{graphicx}
\usepackage{scalerel,stackengine}

\usepackage{mathtools}

\newenvironment{ack}{{\bf Acknowledgements. }}

\newtheorem{theorem}{Theorem}
\newtheorem{proposition}[theorem]{Proposition}
\newtheorem{lemma}[theorem]{Lemma}
\newtheorem{corollary}[theorem]{Corollary}
\theoremstyle{definition}

\theoremstyle{remark}
\newtheorem{remark}[theorem]{Remark}

\definecolor{verde}{RGB}{20,150,100}

\newcommand{\vect}{\overrightarrow}
\newcommand{\nif}{{n \rightarrow +\infty}}
\newcommand{\vv}{{\bf v}}
\newcommand{\ee}{{\bf e}}
\newcommand{\G}{\Gamma}
\newcommand{\Om}{\Omega}
\newcommand{\lb}{\lambda}

\newcommand{\sm}{\setminus}

\newcommand{\sq}{\subseteq}
\newcommand{\ov}{\overline}
\newcommand{\vphi}{\varphi}
\newcommand{\vps}{\varepsilon}
\newcommand{\ra}{\rightarrow}
\newcommand{\R}{\mathbb R}

\stackMath
\newcommand\reallywidecheck[1]{%
\savestack{\tmpbox}{\stretchto{%
  \scaleto{%
    \scalerel*[\widthof{\ensuremath{#1}}]{\kern-.6pt\bigwedge\kern-.6pt}%
    {\rule[-\textheight/2]{1ex}{\textheight}}
  }{\textheight}%
}{0.5ex}}%
\stackon[1pt]{#1}{\scalebox{-1}{\tmpbox}}%
}

\def\R{\mathbb{R}}

\def\1{{{\bf 1}}}

\begin{document}

\title[]{Maximization of the second non-trivial Neumann eigenvalue}

\author[]{Dorin Bucur, Antoine Henrot}
\thanks{{Dorin Bucur is member of the Institut Universitaire de France. His work is part of the  "Geometry and Spectral Optimization" GeoSpec, Persyval Lab research programme.}}

\address[Dorin Bucur]{
Univ. Grenoble Alpes, Univ. Savoie Mont Blanc, CNRS, LAMA \\
73000 Chambéry, France
}
\email{dorin.bucur@univ-savoie.fr}
\address[Antoine Henrot]{
Institut Elie Cartan de Lorraine \\ CNRS UMR 7502 and Universit\'e de Lorraine \\
BP 70239
54506 Vandoeuvre-l\`es-Nancy, France}
\email{antoine.henrot@univ-lorraine.fr}



\begin{abstract}
In this paper we prove that the second (non-trivial) Neumann eigenvalue of the Laplace operator on smooth domains of $\R^N$ with prescribed measure $m$ attains its maximum on the union of two disjoint balls of measure $\frac m2$. As a consequence, the P\'olya conjecture for the Neumann eigenvalues holds for the second eigenvalue and for arbitrary domains. 
We moreover prove that a relaxed form of the same inequality holds in the context of non-smooth domains and densities.  
  \end{abstract}
\maketitle

\section{Introduction and Statement of the results}
Let $N \ge 2$  and $\Om\sq \R^N$ be a bounded open set such that the Sobolev space $H^1(\Om)$ is compactly embedded in $L^2(\Om)$ (for instance $\Omega$ Lipschitz). Those sets are called {\it regular} throughout the paper. On such domains, the spectrum of the Laplace operator with Neumann boundary conditions consists only on eigenvalues that we denote (counting their multiplicities) 
$$0 = \mu_0(\Om)\le \mu _1(\Om) \le \mu_2(\Om) \le \dots \ra +\infty.$$
For every $k \ge 1$, we have
$$\mu_k(\Om) = \min_{S\in{\mathcal S}_k} \max_{u \in S} \frac{\int_\Om |\nabla u|^2 dx}{\int_\Om u^2 dx},$$
where ${\mathcal S}_k$ is the family of all subspaces of dimension $k$ in 
$$H^1(\Om)_{/\R}:=\{u \in H^1(\Om) :\int_\Om u dx =0\}.$$
If $\Om$ is connected, then $\mu_1(\Om) >0$. 

 In 1954 Szeg\"{o} proved  that among simply connected, two dimensional, smooth sets $\Om \sq \R^2$  the ball maximizes $\mu_1$ (see \cite{Sz54} and \cite{As99,BL08})), i.e.\footnote{  Weinberger noted  in \cite{We56} that the proof of Szeg\"{o} gives a stronger result, namely that the disc minimizes the sum $\frac{1}{\mu_1(\Om)}+\frac{1}{\mu_2(\Om)}$ among two dimensional, smooth, simply connected sets of given area.}
$$|\Om|\mu_1(\Om)\le |B|\mu_1(B).$$
Two years later, Weinberger \cite{We56} removed the topological constraint and the dimension restriction and he proved that for every $N\ge 2$ and $\Om \sq \R^N$ regular, the following inequality holds
$$|\Om|^\frac 2N\mu_1(\Om)\le \mu_1^*,$$
where $  \mu_1^*=|B|^\frac 2N\mu_1(B)$. 

Maximizing the Neumann eigenvalues under volume constraint is also related to the celebrated conjecture of P\'olya \cite{Po54b} asserting that the principal term of the Weyl law provides in fact a bound for the eigenvalues. This conjecture reads, in $N$ dimensions,
$$\forall k \ge 1, \;\;\, \mu_k(\Om) \le 4\pi^2 \Big (\frac{k}{\omega_N |\Om|)}\Big )^\frac 2N,$$
where $\omega_N$ is the volume of the unit ball in $\R^N$. The conjecture was proved to hold only for particular classes of domains, for instance tiling domains in $\R^2$ (see \cite{La97}). For general regular domains, the conjecture holds true in the case $k=1$, as a consequence of the Szeg\"{o}-Weinberger inequality, but continues to remain open for arbitrary $k$. Kr\" oger found in  \cite{Kr92} a series of  bounds, which are larger than the conjectured ones. For instance, if $k=2$ he proved  $\mu_2(\Om) \le \frac{16\pi}{|\Om|}$ for two dimensional domains. The value  $ \frac{16\pi}{|\Om|}$ is the double of the conjectured one.  A natural, related,  question is to find the geometry of the domain which maximizes the $k$-th Neumann eigenvalue. This question turns out to be difficult for $k=2$ and probably impossible to answer analytically for $k \ge 3$. We refer to \cite{AO17,AF12, Be15} for numerical approximations of the  (presumably) optimal sets for $k \le 10$, but there is no  proof of the existence of those sets. 

We refer the reader to the result of Girouard, Nadirashvili and Polterovich \cite{GNP09} where the authors prove that in $\R^2$ the union of two equal (and disjoint) disks  gives a larger second eigenvalue than any smooth simply connected open set of the same measure. Moreover, this value is asymptotically attained by two disks with vanishing intersection. Their proof is based on a combination of topological and analytical arguments and relies on a {\it folding and rearrangement} technique introduced by Nadirashvili in \cite{Na02}, taking advantage on the use of conformal mappings. This method can not be adapted to non simply connected sets.  The authors left the case of arbitrary (regular) domains of $\R^2$ as an open problem. 

Several independent numerical computations \cite{AO17,AF12, Be15} in $\R^2, \R^3$ brought support in favor of the maximality of the union of the two discs without the simply connectedness constraint in $\R^2$ and, moreover, lead to a similar conjecture in  three dimensions of the space.

The purpose of this paper is to prove that, in general, the second Neumann eigenvalue attains its maximum  on a union of two disjoint, equal balls in the class of arbitrary domains of prescribed measure of $\R^N$.    As a consequence, we prove that the P\'olya conjecture for Neumann eigenvalues holds for $k=2$, without any restriction on the dimension, geometry or topology of the domains. 

\medskip

Let us denote the scale invariant quantity $\mu_2^*=2^\frac 2N |B|^\frac 2N\mu_1(B)$, where $B$ is any ball.   On the union of two disjoint balls $B_1, B_2$, each of mass $\frac 12$, we have $\mu_0(B_1 \cup B_2)=0, \mu_1(B_1\cup B_2)=0, \mu_2(B_1\cup B_2) = \mu_2^*$. The main result of the paper is the following.
\begin{theorem}\label{bh01}
Let $\Om\sq \R^N$ be a regular set. Then
$$|\Om|^\frac 2N\mu_2(\Om)\le \mu_2^*.$$
 If equality occurs, then $\Om$ coincides a.e. with the union of two disjoint, equal balls.
\end{theorem}
As a consequence, we get the following.
\begin{corollary}\label{bh03.c}
The P\'olya conjecture for the Neumann eigenvalues, holds  for $k=2$ in any dimension of the space.
\end{corollary}
In fact, we shall prove a more general result than Theorem \ref{bh01}. Specifically, we shall prove that the result of Theorem \ref{bh01} holds true on arbitrary open sets (even non regular) and, moreover, on $L^1\cap L^\infty$-densities in $\R^N$, provided the classical eigenvalues, seen as variational quotients, receive a suitable {\it relaxed} definition (see \cite[Chapter 7]{BB05} and \cite{BBG17}). 

Precisely, let $\rho \in L^1(\R^N, [0,1])$. For every $k\ge 1$, we define
$$\tilde \mu _k(\rho) := \inf_{S\in{\mathcal L}_k} \max_{u \in S} \frac{\int_{\R^N} \rho|\nabla u|^2 dx}{\int_{\R^N} \rho u^2 dx},$$
where ${\mathcal L}_k$ is the family of all subspaces of dimension $k$ in 
\begin{equation}\label{bh08}
\{u\cdot 1_{\{\rho (x)>0\}}: u \in C^\infty_c (\R^N), \int_{\R^N} \rho u dx =0\}.
\end{equation}
We have the following.
\begin{theorem}\label{bh03}
Let $\rho \in L^1(\R^N, [0,1])$ be non identically zero. Then
\begin{itemize}
\item ($k=1$, extension of the Szeg\"{o}-Weinberger inequality)
\begin{equation}\label{bh09}
 \left(\int_{\R^N} \rho dx \right)^\frac 2N\tilde \mu_1(\rho)\le \mu_1^*,
 \end{equation}
with equality if and only if $\rho= 1_B$ a.e., for some ball $B$ of $\R^N$.
\item $(k=2)$
\begin{equation}\label{bh10}
\left(\int_{\R^N} \rho dx\right)^\frac 2N\tilde \mu_2(\rho)\le \mu_2^*,
\end{equation}
with equality if and only if $\rho= 1_{B^\sharp}+1_{B^*}$ a.e., where $B^\sharp, B^*$ are two disjoint (open) balls of equal measure. 
\end{itemize}
\end{theorem}
For $k=1$, Theorem \ref{bh03} above is  a generalization of the Szeg\"{o}-Weinberger inequality, and for $k=2$ is a generalization of Theorem \ref{bh01}. 

We notice the following.
\begin{itemize}

\item If $\Om$ is a bounded, open Lipschitz set, then taking $\rho= 1_\Om$, one gets $\tilde \mu_k(\rho)= \mu_k(\Om)$. 

\item Let us remove a smooth manifold $\Gamma$ from $\Omega$, such that $H^1(\Om \sm \G)$ is compactly embedded in $L^2(\Om\sm \G)$.
This is for example the case when $\Omega$ is Lipschitz and the crack $\Gamma$ is itself Lipschitz. Then for $\rho = 1_\Om$ one has 
$\tilde \mu_k(\rho)\ge  \mu_k(\Om\sm \G)$ since $C^\infty_c(\R^N)|_{\Om\sm \G} \sq H^1(\Om\sm \G)$. From this perspective, Theorem \ref{bh03} covers the inequality proved in Theorem \ref{bh01} even in this less regular case. 

\item If the set $\Omega=\{\rho >0\}$ is smooth and there exists $\alpha >0$ such that $\rho \ge \alpha 1_{\{\rho >0\}}$ (i.e. $\rho $ is not degenerating on its support and preserves ellipticity), then $\tilde \mu_k (\rho)$ are the eigenvalues associated to the well posed problem
$$-div(\rho \nabla u)=\mu_k \rho u \;\;\mbox {in}\;\Om, \;\;\frac{ \partial u}{\partial n} =0\;\;\mbox {on}\;\partial \Om.$$

\item If $\Om$ is just a bounded open set, without any smoothness, the spectrum of the Neumann Laplacian on $\Om$ may be continuous. Theorem \ref{bh03} still applies to $\rho=1_\Om$, but we do not have any spectral interpretation of $\tilde \mu_k(1_\Om)$. The same occurs if either $\rho$ is degenerating loosing ellipticity on its support, and/or if its support is not smooth enough. 
\end{itemize}
Concerning the ideas of the proof, it is worth to recall what happens for the Dirichlet Laplacian. The Faber-Krahn inequality for the first Dirichlet eigenvalue of the Laplacian $\lb_1(\Om)$ asserts that the minimum of $|\Om|^\frac 2N\lb_1(\Om)$ is attained on balls. A simple argument analysing the positive and negative parts of a second eigenfunction, leads to the conclusion that the minimum of $|\Om|^\frac 2N\lb_2(\Om)$ is achieved on a set consisting on two equal and disjoint balls. We refer to \cite{BF13} for a detailed description of the history of the result, which is attributed to Krahn \cite{Kr25}, Hong \cite{Ho54}    and Szeg\"{o} \cite{Po55}. 

A similar argument for the Neumann Laplacian is not valid. 
The proof of Theorem \ref{bh01} (and of Theorem \ref{bh03}) is based on a suitable construction of a set of $N$ test functions which are simultaneously orthogonal to the constant function {\it and} to the first Neumann eigenfunction on a regular set $\Om$. The structure of these $N$-functions is inspired by the functions built by Weinberger, that we briefly describe  below (see, for instance, \cite{He06}, \cite{We56} for more details).

We denote throughout the paper $R_\Om$ the radius of a ball of the same volume as $\Om$, by $r_\Om$ the radius of a ball of volume $\frac{|\Om|}{2}$, by $B_R$ a ball centered at the origin of radius $R$, and by $B_{A,R}$ the ball centered at point $A$ of radius $R$. We denote by $g$ a non negative, strictly increasing 
solution\footnote{The function $g$ is explicitly given by $g(r)=r^{1-N/2} J_{N/2}(kr/R)$ where $k=\sqrt{\mu_1(B_R)}$ is the first positive zero of $r\mapsto [r^{1-N/2} J_{N/2}(r)]'$ sometimes denoted $p_{N/2,1}$ as in \cite{AB} and 
$J_{N/2}$ is the standard Bessel function. } of the following differential equation on $(0,R)$ (see the paper of Weinberger \cite{We56}, or \cite[Section 7.1.2]{He06} for details)
\begin{equation}\label{bh11.9}
g''(r)+\frac {N-1}{r} g'(r) +  (\mu_1(B_R)-\frac{N-1}{r^2})g(r)=0, g(0)=0, g'(R)=0.
\end{equation}

 Given a point  $A \in \R^N$ and a value $R>0$, Weinberger introduced the function
\begin{equation}\label{bh27}
{\bf g}_A: \R^N \ra \R^N\;\;   {\bf g}_A (x) = \frac{G_R(d_A(x))}{d_A(x)}\vect {Ax},
\end{equation} 
where $G_R:[0,+\infty) \ra \R$,
\begin{equation}\label{bh11}
G_R(r):=g(r)1_{[0, R]} +  g(R) 1_{[R, +\infty)}. 
\end{equation}
 By   $d_A(x)$  we denoted the  distance from $x$ to $A$.

 Using Brouwer's fixed point theorem, Weinberger proved for $R=R_\Om$  the existence of a point $A$ such that the set of functions
$$x \mapsto {\bf g}_A(x)\cdot \ee_i, \; i=1, \dots, N$$ 
are orthogonal to the constants in $L^2(\Om)$. Here $(\ee _i)_i$ are the vectors of an orthonormal basis. As a consequence, those functions can be taken as tests in the Rayleigh quotient for $\mu_1(\Om)$. By summation this lead to
\begin{equation}\label{bh25}
\mu_1(\Om) \le \frac{\int_{ \Om} G_{R_\Om}'^2(d_A(x))+ (N-1) \frac{G_{R_\Om}^2(d_A(x))}{d_A^2(x)} dx}{\int_{\Om} G_{R_\Om}^2(d_A(x)) dx}.
\end{equation}
The function $r \mapsto G_{R_\Om} (r)$ is strictly increasing on $[0, R_\Om]$ (and then constant), while 
$$r \mapsto G_{R_\Om}'^2(r) + (N-1) \frac{G_{R_\Om}^2}{r^2}$$
is decreasing. Consequently,  the right hand side of \eqref{bh25} is not larger than
\begin{equation}\label{bh26}
\mu_1(B_{A,R_\Om})= \frac{\int_{ B_{A,R_\Om}} G_{R_\Om}'^2(d_A(x))+ (N-1) \frac{G_{R_\Om}^2(d_A(x))}{d_A^2(x)} dx}{\int_{B_{A,R_\Om}} G_{R_\Om}^2(d_A(x)) dx}.
\end{equation}
In order to observe that $\mu_1(\Om) \le \mu_1(B_{A,R_\Om})$, one has formally  to move, one to one, the points of $\Om\sm B_{A,R_\Om}$ toward the points of the $B_{A,R_\Om} \sm \Om$  pushing forward the measure $1_{\Om\sm B_{A,R_\Om}} dx$ to $1_{B_{A,R_\Om} \sm \Om} dx$. Throughout the paper, we call  this procedure {\it a mass displacement argument}.   

\medskip
In order to prove Theorem \ref{bh01}, we shall use a somehow similar strategy, searching a set of $N$ suitable test functions. The new difficulty is that the set of functions that we have to build, should be orthogonal  to both the constant function and to a first Neumann eigenfunction on $\Om$ (which is unknown). In the same time, the associated Rayleigh quotient should not exceed $\mu_2^*$.

 Given two different points $A,B \in \R^N$, we introduce the linear part of the symmetry operator with respect to the mediator hyperplane ${\mathcal H}_{AB}$ of the segment $AB$
 $$T_{AB}: \R^N \ra \R^N, \;\; T_{AB}(\vv)= \vv - 2 (\vect{ab}\cdot \vv) \vect{ab},$$
where $\vect{ab}= \frac{\vect{AB}}{\|AB\|}$.
Denoting $H_A$, $H_B$ the half spaces determined by ${\mathcal H}_{AB}$ and containing $A$ and $B$, respectively, we build the function
\begin{equation}\label{bh17aa}
 {\bf g}^{AB} : \R^N \ra \R^N, \;\;   {\bf g}^{AB}(x)= 1_{H_A} (x){\bf g}_A(x)+ 1_{H_B}(x)T_{AB} ({\bf g}_B(x)).
 \end{equation}
 The functions ${\bf g}_A$, ${\bf g}_B$ are the functions of Weinberger introduced in \eqref{bh27}, associated  to $G_{r_\Om}$. Roughly speaking, $ {\bf g}^{AB}$ is a suitable gluing along ${\mathcal H}_{AB}$ of two Weinberger functions  corresponding to different points. 
\begin{figure}[ht]
\includegraphics[width=11cm]{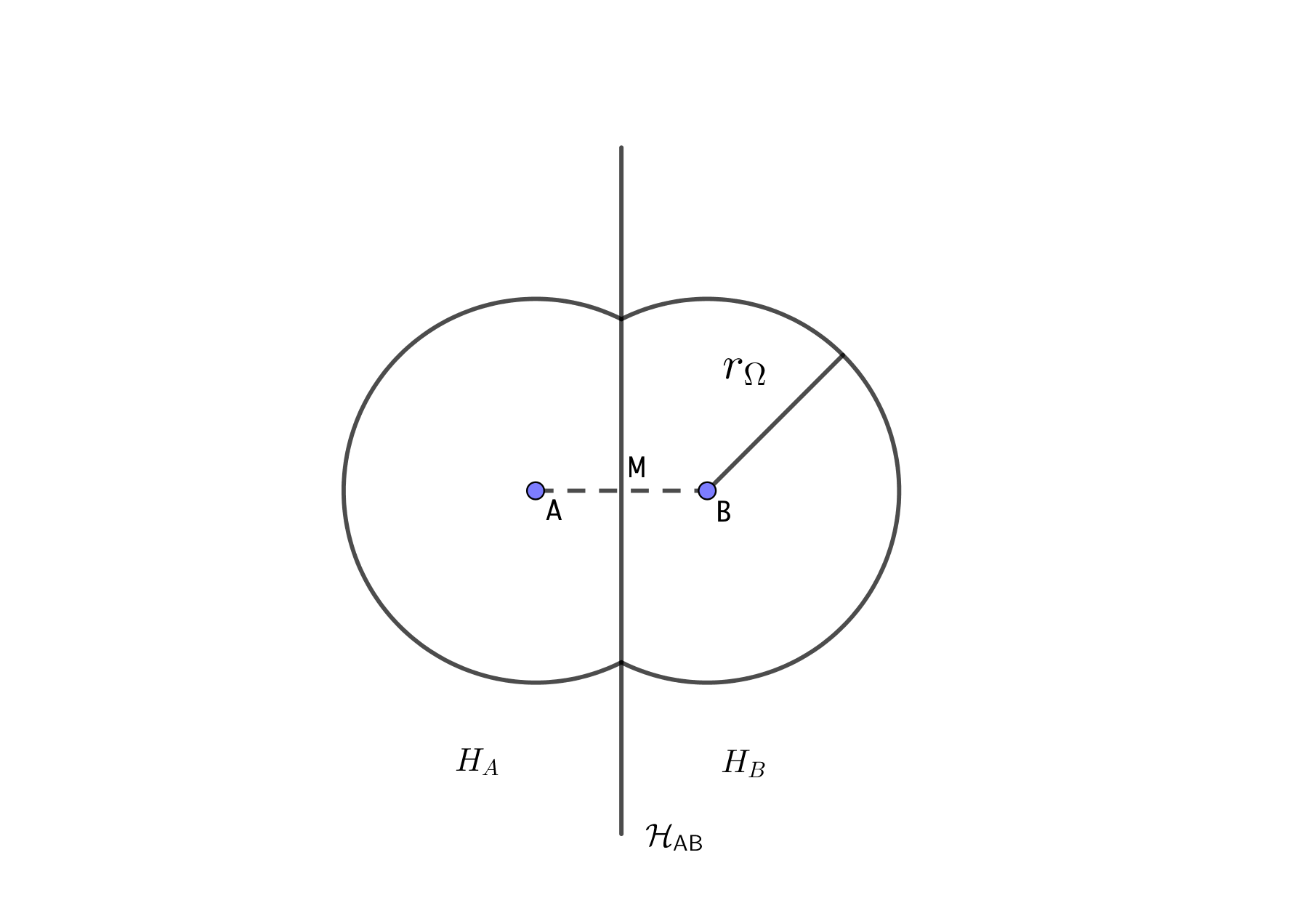} 
\caption{The geometry of the test functions ${\bf g}^{AB}$.}   
\label{bhfig1}   
\end{figure}

Our main purpose will be to justify the existence of two points $A,B$ such that the set of $N$ scalar functions
 $$x \mapsto {\bf g}^{AB}(x)\cdot {\ee_i}, i=1, \dots, N,$$ 
are simultaneously orthogonal in $L^2(\Om)$ on the constant function and on a first eigenfunction $u_1$ of the Neumann Laplacian on $\Om$, i.e.
\begin{equation}\label{bh16}
\forall i=1, \cdots, N, \;\; \int_\Om {\bf g}^{AB}\cdot {\ee_i}  dx= \int_\Om {\bf g}^{AB}\cdot {\ee_i} u_1 dx=0.
\end{equation}
  The proof of existence of $A$ and $B$ with such properties relies on {\bf a topological degree argument} and requires the most of the attention (it is worth mentioning that every result on maximization for Neumann eigenvalues in the literature relies on a strong topological argument). 

 Once the points $A$ and $B$ are found, the proof of Theorem \ref{bh01} follows in its main lines the one of Weinberger, being based on the mass displacement argument. In fact, on each of the half spaces $H_A,H_B$, the restriction of ${\bf g}^{AB}$ acts like a Weinberger function \eqref{bh27} associated to a ball of half measure.

\vspace{5mm}
\noindent {\bf Structure of the paper.} 
\begin{itemize}
\item In the next section we prove Theorem \ref{bh01} for regular sets. 
We give the detailed construction of the  function ${\bf g}^{AB}$, prove the existence of a couple of points $A$ and $B$ making ${\bf g}^{AB}\cdot {\ee_i}$  suitable as test functions for $\mu_2(\Om)$, and prove the inequality. The equality case will be a direct consequence. 
\item In section \ref{s3}, we give the proof of Theorem \ref{bh03}.  We start by proving that the classical Szeg\"{o}-Weinberger inequality holds true as well for densities and arbitrary domains. Concerning the second eigenvalue, we rely on the main ideas introduced in Section \ref{s2} and focus on the new difficulties raised by the possible absence of a first eigenfunction and by the possible unboundedness of the support.
\end{itemize}
Although it would have been more natural to prove first the general case and deduce the inequality for regular domains as a consequence, we have chosen to start by giving a detailed proof of Theorem \ref{bh01} in the classical framework, as most of  readers are presumably interested in this case. The new difficulties raised by the proof of  Theorem \ref{bh03} are exclusively related to the ill-posedness of the eigenvalue problem in the non-smooth/degenerate/unbounded setting. The fact that we deal with a density instead of a geometric domain does not raise any supplementary difficulty, being handled  by mass displacement.

\section{Proof of Theorem \ref{bh01}}\label{s2}
In this section we prove Theorem \ref{bh01}. Let $\Om \sq \R^N$ be regular. We split the proof in several parts.

\medskip
\noindent {\bf The validity of  the test functions.} Recall that $r_\Om$ is the radius of the ball of volume $\frac{|\Om|}{2}$.
 A set of eigenfunctions  associated to the first non-zero eigenvalue $\mu_1(B_{r_\Om})$ on the ball $B_{r_\Om}$ are $\{ \frac{g(r)}{r} x_i: i=1, \dots, N\}$, where $g$ solves the differential equation \eqref{bh11.9} for $R=r_\Om$.

Let $A,B$ be two distinct points in $\R^N$. We recall the function  ${\bf g}^{AB}$ introduced in \eqref{bh17aa}, 
 $${\bf g}^{AB}(x)= 1_{H_A}(x) {\bf g}_A(x)  + 1_{H_B}(x){\bf g}_B(x) - 2 \cdot 1_{H_B}(x)({\bf g}_B(x)\cdot \vect{ab}) \vect{ab}.$$
The function ${\bf g}^{AB}$ is well defined, and continuous across ${\mathcal H}_{AB}$. Indeed, it is enough to observe that for
 $x_0 \in \partial H_A \cap \partial H_B={\mathcal H}_{AB}$ we have
 $${\bf g}_A(x_0)=  T_{AB} ({\bf g}_B(x_0)),$$
 which comes from direct computation.

We notice that $\forall x \in \R^N$, 
$$\|{\bf g}^{AB}(x)\|\le g(r_\Om)$$
\begin{equation}\label{boundgrad}
\|\nabla {\bf g}^{AB}(x)\| \le 2N^2 \big (\sup_{r\in (0, +\infty)} \frac {G_{r_\Om}(r)}{r}+ \sup_{r\in (0, +\infty)} G'_{r_\Om}(r)\big )<+\infty.
\end{equation}
As a conclusion,  we get ${\bf g}^{AB}\in W^{1, \infty} (\R^N, \R^N)$, with a uniform bound on their norm, independent on $A$ and $B$. 
The functions ${\bf g}^{AB}\cdot \ee_i$ are therefore admissible as test functions on $\Om$. 

\medskip
\noindent {\bf The use of the test functions.} Assume for a moment that we have found two different points $A, B \in \R^N$ such that
the orthogonality relations \eqref{bh16} hold. 
Let us show that we can prove Theorem \ref{bh01}. 

For some $i\in \{1, \dots, N\}$, let us take in the definition of $\mu_2(\Om)$ the subspace $S=span \{ u_1, {\bf g}^{AB}\cdot {\ee_i} \}$. 
We can write
$$\forall i=1, \dots, N, \;\; \mu_2(\Om) \le \frac{\int_\Om |\nabla ({\bf g}^{AB}\cdot {\ee_i})|^2 dx}{\int_\Om |{\bf g}^{AB}\cdot {\ee_i}|^2 dx}.$$
As a consequence, 
$$\mu_2(\Om) \le \frac{\sum_{i=1}^N\int_\Om |\nabla ({\bf g}^{AB}\cdot {\ee_i})|^2 dx}{\sum_{i=1}^N \int_\Om |{\bf g}^{AB}\cdot {\ee_i}|^2 dx}.$$
Decomposing the sums over $\Om \cap H_A$ and $\Om\cap H_B$, we get using \eqref{bh17aa} (the computation is similar with the one in Weinberger's proof, see \cite{He06}))
$$\mu_2(\Om) \le \frac{\int_{H_A\cap \Om} G_{r_\Om}'^2(d_A(x))+ (N-1) \frac{G_{r_\Om}^2(d_A(x))}{d_A^2(x)} dx+\int_{H_B\cap \Om} G_{r_\Om}'^2(d_B(x)) + (N-1) \frac{G_{r_\Om}^2(d_B(x))}{d_B^2(x)}dx}{\int_{H_A\cap \Om} G_{r_\Om}^2(d_A(x)) dx+\int_{H_B\cap \Om} G_{r_\Om}^2(d_B(x)) dx }.$$
We displace the mass as follows: we split $\Om \sm (B_{A, r_\Om}\cup B_{B, r_\Om})$ in two sets $\Om_A$ and $\Om_B$, such that 
$|\Om_A|+|\Om \cap B_{A, r_\Om}|= \frac {|\Om|}{2}$. By monotonicity of $r\mapsto G_{r_\Om}(r)$ and of $r\mapsto G_{r_\Om}'^2(r)+ (N-1) \frac{G_{r_\Om}^2(r)}{r^2}$, for any couple of points  $x\in \Om_A$ and $y \in B_{A, r_\Om}\sm \Om$  the following inequalities hold
$$\frac{G_{r_\Om}'^2(d_A(x))+ (N-1) \frac{G_{r_\Om}^2(d_A(x))}{d_A^2(x)}}{d_A^2(x)}< \frac{G_{r_\Om}'^2(d_A(y))+ (N-1) \frac{G_{r_\Om}^2(d_A(y))}{d_A^2(y)}}{d_A^2(y)}$$
$$G_{r_\Om}^2(d_A(x)) > G_{r_\Om}^2(d_A(y)).$$
We then formally   displace the mass  from $\Om_A$ to $B_{A, r_\Om}\sm \Om$, and increase the Rayleigh quotient. We use the same argument for  $\Om_B$ and $B_{B, r_\Om}\sm \Om$, finally getting that
  $$\frac{\int_{H_A\cap \Om} G_{r_\Om}'^2(d_A(x))+ (N-1) \frac{G_{r_\Om}^2(d_A(x))}{d_A^2(x)} dx+\int_{H_B\cap \Om} G_{r_\Om}'^2(d_B(x)) + (N-1) \frac{G_{r_\Om}^2(d_B(x))}{d_B^2(x)}dx}{\int_{H_A\cap \Om} G_{r_\Om}^2(d_A(x)) dx+\int_{H_B\cap \Om} G_{r_\Om}^2(d_B(x)) dx }$$
  $$\le \frac{2 \int _{B_{r_\Om}} \big (G_{r_\Om}'^2(r) + (N-1) \frac{G_{r_\Om}^2}{r^2}\big )dx}{2 \int _{B_{r_\Om}} G_{r_\Om}^2(r) dr}= \mu_1(B_{r_\Om}).$$
 Since $ \mu_1(B_{r_\Om})$ is the second eigenvalue of the union of two disjoint balls of mass  $\frac {|\Om|}{2}$, the inequality in Theorem \ref{bh01} follows.
 
 If equality occurs, then $H_A\cap \Om$ and $H_B\cap \Om$ have to be balls of mass $\frac {|\Om|}{2}$ up to a set of zero Lebesgue measure. Indeed,  if there is mass displacement on a set of positive measure, the inequality has to be strict. So, if equality occurs,  $\Om$ is a.e. identical to the union of two disjoint balls of mass $\frac {|\Om|}{2}$.
 \begin{remark}
 If, for instance $\Om$ is Lipschitz and equality occurs, then $\Om$ has to coincide with the union of the two balls. If we work only with regular sets without further geometric assumption, it might be possible that one removes from one ball  a set of capacity zero and, from the other ball, a (small) set of positive capacity but of zero measure  (say a piece of a smooth manifold of dimension $N-1$). The removed  set, should be small enough such that the second non-trivial eigenvalue of the slitted ball is not smaller than the first eigenvalue of the genuine ball. In $\R^2$, this situation could occur if one removes a small segment from a diameter.
 \end{remark}

\medskip
\noindent {\bf Existence of the family of test functions.} In order to complete the proof, it remains to justify the existence of two points $A,B$ such that the orthogonality relations \eqref{bh16} hold true. We shall do this below, but we point out from the  beginning that the proof works in an identical way provided $1_\Om$ is replaced by a measurable function $\rho :\R^N\ra [0,1]$ with {\it  bounded} support, and the first eigenfunction $u_1$ is replaced by any measurable  function $u$ such that $u 1_{\{\rho=0\}}=0$ and $\int_{\R^N}\rho u^2  dx <+\infty$ and $\int_{\R^N}\rho u dx =0$. 

We give the following.
\begin{lemma}\label{bh20}
Let $A\not = B$ two points of $\R^N$. Then, for all $x \in \R^N$
$$ {\bf g}_A (x) \cdot \vect {ab} >  {\bf g}_B (x) \cdot \vect {ab}.$$
\end{lemma}
\begin{proof}
The proof is immediate, by direct comparison.

\end{proof}
\begin{lemma}\label{bh21}
Assume that  $A, B$ are two points of $\R^N$ such that
$$\forall i=1, \dots, N, \;\; \int_\Om  {\bf g}_A (x) \cdot \ee_i dx = \int_\Om  {\bf g}_B (x) \cdot \ee_i dx.$$
Then  for all $\vv \in \R^N$ we have 
$$\int_\Om  {\bf g}_A (x) \cdot \vv dx = \int_\Om  {\bf g}_B (x) \cdot \vv dx,$$
and
$A=B$. 
\end{lemma} 
\begin{proof}
The first assertion is trivial and the second is a consequence of Lemma \ref{bh20} for $\vv =  \vect {ab} $.
\end{proof}

In the sequel, we shall use a deformation argument in the framework of  the topological degree theory (see for instance \cite[Theorem 1]{Br83}), in order to prove the  following.
\begin{proposition}\label{bh22}
There exist two different points $A,B$ such that the orthogonality relations \eqref{bh16} hold true.
\end{proposition}
\begin{proof}
By rescaling, we may assume that $\Om \sq B_1$, the ball centered at the origin of radius equal to $1$. 
Let $M\ge 20$ be fixed (the value $20$ is chosen to be large enough with respect to the radius of  $B_1$). Denote 
$${\mathcal D}=\{ (X,Y): X,Y \in \R^N, X=Y\}\sq \R^{2N}.$$
We introduce the function
$$F: [-M,M]^{2N}\sm {\mathcal D} \ra \R^{2N},$$
by
$$F(A,B):= \big ( \int_\Om {\bf g}^{AB} \cdot \ee_idx, \int_\Om {\bf g}^{AB} \cdot \ee_i  u_1dx \big ).$$
We want to prove that there exists a couple of  points $(A,B) \in [-M,M]^{2N}\sm {\mathcal D}  $ which make  $F$ vanish. So we assume for contradiction that $F$ does not vanish on its definition domain. We first observe that there exists $\delta >0$ such that if $(A,B) \in [-M,M]^{2N}\sm {\mathcal D}$ and 
$$d_{\R^{2N}} ((A,B), {\mathcal D}) \le \delta,$$
then $F$ can not vanish at $(A,B)$. Indeed, assume for contradiction that $(A_n,B_n) \in [-M,M]^{2N}\sm {\mathcal D}$ is such that 
$$\forall i=1, \dots, N \;\int_\Om {\bf g}^{A_nB_n} \cdot \ee_idx =0,$$
and $d_{\R^{2N}} ((A_n,B_n), {\mathcal D}) \ra 0$. Extracting a subsequence, we can assume that $A_n\ra A$, $B_n \ra A$, $\frac{\vect{A_nB_n}}{\|A_nB_n\|} \ra \vv\in S^{N-1}$.
Then the a.e. limit of the sequence of functions $({\bf g}^{A_nB_n}\cdot \vv_n)_n$, denoted for convenience ${\bf g}^{AA}\cdot \vv$, has a constant sign, vanishing only on a zero measure set. This contradicts $ \int_\Om {\bf g}^{A_nB_n} \cdot \vv_ndx =0$. 

So, let us denote 
$$V= \{ (A,B) \in \R^N\times \R^N : d_{\R^{2N}} ((A,B), {\mathcal D}) \le \delta\},$$
and restrict the function $F$ to $[-M,M]^{2N}\sm V$. 

Let $B_{X^*,R}$ be a ball of some radius $0<R\le 1$ (the choice is free, but we should have in mind $r_\Om$), with a center $X^*$ carefully chosen, that will be specified in the proof. For simplicity of the notation, we denote this ball $B^*$.

\medskip
\noindent{\bf First deformation.}
We introduce for $t\in [0,1]$ the following family of functions
$$F_t : [-M,M]^{2N}\sm V \ra \R^{2N},$$
by
$$F_t (A,B)= \big ( \int_\Om {\bf g}^{AB} \cdot \ee_idx, (1-t) \int_\Om {\bf g}^{AB} \cdot \ee_i  u_1dx + t \int_{B^*}{\bf g}^{AB} \cdot \ee_i  dx \big ).$$
Clearly, this family is continuous in $t$, and $F_0\equiv F$. We shall prove that for a specific position of the center  $X^*$ of the ball $B^*$, for every $t \in[0,1]$ the function $F_t$ can not vanish on 
$\partial ([-M,M]^{2N}\sm V)$. 

Assume that for some $(A,B) \in \partial ([-M,M]^{2N}\sm V)$ and some $t\in (0,1]$ we have  $F_t(A,B)=0$. We shall focus first only on the first $N$ coordinates of $F_t(A,B)$ which depend neither on $t$ nor on $B^*$. This will give important information on the possible positions of the points $(A,B)$. 

Indeed, we start observing that $(A,B) \not \in \partial V$. This is a consequence of the choice of $\delta$, above. It remains that $(A,B) \in \partial ([-M,M]^{2N})$. In other words, at least one of the points $A$ or $B$ is at distance at least $M$ from the origin (hence at least $M-1$ from $\Om$). 

\medskip
\noindent {\bf Case 1. Assume the point $B$ is at distance at least $M-1$ from $\Om$.} Let ${\mathcal C}$ be the cone with vertex $B$, tangent to the ball $B_1$. If the point $A$ does not belong to the cone ${\mathcal C}$, then denoting $O'$ the projection of $O$ on the line $AB$, we can not have
$$\int_{\Om} {\bf g}^{AB} \cdot \vect{O'O} dx =0$$
since the function ${\bf g}^{AB} \cdot \vect{O'O}$ has constant sign on $\Om$. So $A$ has to belong to the cone. Moreover, in this situation, the point $A$ has to belong as well to the annulus $B_{B,M+1}\sm B_{B,M-1}$. Indeed, if $A$ does not belong to this annulus 
we can not have
$$\int_{\Om} {\bf g}^{AB} \cdot \vect{BO} dx =0$$
since, this time, the function ${\bf g}^{AB} \cdot \vect{BO}$ has constant sign on $\Om$ (positive if
$A$ is between the ball $B_1$ and $B$, negative if the ball $B_1$ is between $A$ and $B$. This means that $A \in {\mathcal C} \cap (B_{B,M+1}\sm B_{B,M-1})$, which by simple computation leads to $A \in  B_{\sqrt{2}}$. 

The main consequence is that the distance from $A$ to $O$ is not larger than $\sqrt{2}$ hence, by the construction of the function ${\bf g}^{AB}$, its action on the domain $\Om$ is entirely given by $A$ since $\Om$ is covered by $H_A$  only. In other words, the point $B$ does not influence the integrals in \eqref{bh16} and ${\bf g}^{AB}={\bf g}_A$ on $\Om$. Moreover, from Lemma \ref{bh21}, we get that the position of $A$ satisfying $F_t(A,B)=0$ is uniquely determined, for every $B$ far away from $\Om$. 

For $B$ far away and $A$ fixed as above, let us look now to the linear form
$$\vv \stackrel{L}{\mapsto} \int_\Om {\bf g}^{AB} \cdot \vv  u_1dx.$$
This form is not identically vanishing, otherwise for the couple $(A,B)$ the function $F$ vanishes. Consequently, the kernel of this form is an hyperplane, denoted ${\mathcal K}$. Let $\xi \in S^{N-1}$ be orthogonal to ${\mathcal K}$ such that $\int_\Om {\bf g}_A \cdot \xi u_1 dx >0$. We choose the center of the ball $X^*$ to be given by 
$A+3 \sqrt{2} \xi$. With this choice, the ball $B^*$ does not intersect $B_1$ and is fully covered by $H_A$ (recall that $M \ge 20$). Consequently,
$$\vv \mapsto \int_{B^*}  \frac{G_R(d_A(x))}{d_A(x)} \vect {Ax} \cdot \vv dx $$
has the same kernel ${\mathcal K}$ and has the same sign as $L$. In other words, for every $t \in [0,1]$
$$\vv \mapsto (1-t)\int_\Om {\bf g}^{AB} \cdot \vv  u_1dx+ t \int_{B^*}  \frac{G_R(d_A(x))}{d_A(x)} \vect {Ax} \cdot \vv dx $$
vanishes only for $\vv \in {\mathcal K}$. 

At least one of the vectors $\ee_1, \dots, \ee_N$ does not belong to ${\mathcal K}$. Consequently, among the $N$ terms
$$ i=1, \dots, N,\;\;\; (1-t)\int_\Om {\bf g}^{AB} \cdot \ee_i  u_1dx+ t \int_{B^*}  \frac{G_R(d_A(x))}{d_A(x)} \vect {Ax} \cdot \ee_i dx,$$
at least one is not vanishing. \\
The conclusion is that for every $t \in [0,1]$ the function $F_t$ can not vanish on $ \partial ([-M,M]^{2N}\sm V)$.

\medskip
\noindent {\bf Case 2. Assume the point $A$ is at distance at least $M-1$ from $\Om$.} In this case, only the point $B$ acts on $\Om$, as in the previous case, and $B\in B_{\sqrt{2}}$. The only thing which differs, is the expression of the function $F_t$, which becomes
$$F_t (A,B)= \big ( \int_\Om {\bf g}_{B} \cdot \ee_i - 2 ({\bf g}_{B} \cdot \vect{ab}) (\vect{ab}\cdot \ee_i)dx, \hskip 6cm $$
$$\hskip 2cm  (1-t) \int_\Om {\bf g}_{B} \cdot \ee_i  u_1  - 2 ({\bf g}_{B} \cdot \vect{ab}) (\vect{ab}\cdot \ee_i)u_1dx+ t \int_{B^*}{\bf g}_{B} \cdot \ee_i  - 2 ({\bf g}_{B} \cdot \vect{ab}) (\vect{ab}\cdot \ee_i)\big ).$$
In other words, we have
$$F_t (A,B)= \big ( \int_\Om {\bf g}_{B} \cdot T_{AB}(\ee_i)dx, (1-t) \int_\Om {\bf g}_{B} \cdot T_{AB}(\ee_i )u_1dx+ t \int_{B^*}{\bf g}_{B} \cdot T_{AB}(\ee_i)   \big ).$$
Again, as in Case 1, if $F_t(A,B)=0$, then $B$ has to coincide with the same point $A$, in the preceding case. 

For the point $X^*$ introduced in Case 1, we can not have for all $\ee_i$
$$(1-t)\,\int_\Om {\bf g}_{B} \cdot T_{AB}(\ee_i )u_1dx+ t \int_{B^*}{\bf g}_{B} \cdot T_{AB}(\ee_i) =0$$
since the range of $T_{AB}$ is of dimension $N$. 
\end{proof}
At that stage, we have proved that the topological degrees of $F_0$ and $F_1$ coincide: ${\rm d}(F_0,[-M,M]^{2N}\sm V,0)=
{\rm d}(F_1,[-M,M]^{2N}\sm V,0)$. We are now going to consider a second continuous deformation which will further simplify the functional.

\medskip
\noindent{\bf Second deformation.} Let $B^\sharp$ the ball obtained by symmetry of $B^*$ with respect to the origin. We define the following continuous deformation
$$G: [0,1] \times\R^{2N} \ra \R^{2N},$$
$$G_t(A,B)= \big ( (1-t) \int_\Om {\bf g}^{AB} \cdot \ee_i dx + t\int_{B^\sharp} {\bf g}^{AB} \cdot \ee_i dx, \int_{B^*}  {\bf g}^{AB} \cdot \ee_i dx \big ).$$
Similarly to Case 1, we do not vary the last $N$ coordinates of $G_t$, which are of the same nature as the first $N$ coordinates of $F_t
$. Consequently, if $G_t(A,B)=0$ for some $t$ and one of the point $(A,B)\in \partial ([-M,M]^{2N}\sm V)$ then 
the other one has to be the center of $B^*$ which is the only point which satisfies 
$\forall \vv,\, \int_{B^*}  {\bf g}^{AB} \cdot \vv dx=0$. Note that
we possibly decrease the value $\delta$ in the computation of $V$, such that $V$ is also suitable for $B^*$. Taking $\vv$ a unit vector parallel with the line $X^*O$, one can notice that  
$$(1-t) \int_\Om {\bf g}_{X^*} \cdot \vv dx + t\int_{B^\sharp} {\bf g}_{X^*} \cdot \vv dx$$
can not vanish since both integrals have the same sign. 

As $G_0=F_1$, we can glue the two continuous deformations and notice that they do not vanish on $\partial ([-M,M]^{2N}\sm V)$, so in view of \cite[Theorem 1]{Br83} they have the same topological degree:
$${\rm d}(F,[-M,M]^{2N}\sm V,0)= {\rm d}(F_1,[-M,M]^{2N}\sm V,0)= {\rm d}(G_1,[-M,M]^{2N}\sm V,0).$$

We shall compute in the sequel the topological degree of $G_1$ and we shall prove that it equals to $2$. As a consequence $F$ has at least one zero, so we conclude the proof.

\medskip
\noindent{\bf Computation of the topological degree of $G_1$ at $0$.} There are two steps.

\smallskip
\noindent{\bf Step 1.The zeros of the function $G_1$.} 
We can assume without loosing generality that the center of $B^*$ is $X^*=(3\sqrt{2}, 0, \dots, 0)$ and the center of $B^\sharp$ is $X^\sharp=(-3\sqrt{2}, 0, \dots, 0)$.
We claim that the only zeros of the function $G_1$ are the couples $(X^*,X^\sharp)$ and $(X\sharp, X^*)$. Assume $A$ and $B$ are such that 
$G_1(A,B)=0$. Then
$$\forall \vv \in \R^N \;\; \int_{B^\sharp} {\bf g}^{AB} \cdot \vv dx = \int_{B^*} {\bf g}^{AB} \cdot \vv dx=0.$$
Assume for contradiction that $X^*\not \in AB$. Denoting $X'$ the projection of $X^*$ on the line $AB$, and taking $\vv = \vect{X'X^*}$ then
$$\int _{B^*} {\bf g}^{AB} \cdot \vv dx \not=0,$$
as a consequence of the structure of the function ${\bf g}^{AB}$ and the symmetry of the ball. Indeed, the function ${\bf g}^{AB} \cdot \vv$ is odd with respect to the hyperplane containing the line $AB$ and orthogonal to $\vv$ and has constant sign on each half space defined by this hyperplane. As this hyperplane does not cut the ball into two half balls, the integral $\int _{B^*} {\bf g}^{AB} \cdot \vv dx$ can	not vanish. Consequently $X^* \in AB$, and similarly $X^\sharp \in AB$. 

Let us denote by $x_A$ and $x_B$ the abscissa of points $A$ and $B$ and $x_M=(x_A+x_B)/2$ the abscissa of the middle.
We discuss with respect to the possible values of $x_M$.
\begin{itemize}
\item If $x_M\in [-3\sqrt{2}+R, 3\sqrt{2}-R]$ each ball is completely contained in one of the two half-spaces,
consequently using the uniqueness given by Lemma \ref{bh21} we get that the two points have to coincide with the two
centres of the balls.
\item Let us prove that it is not possible that $x_M\in ]-\inf, -3\sqrt{2}+R[ \cup ]3\sqrt{2}-R,+\infty[$. Indeed, in that case
the two balls would be in the same half-space and the uniqueness result of Lemma \ref{bh21} would imply
that the same point $A$ or $B$ should be the center of each ball.
\item At last if $x_M\in [-3\sqrt{2}-R, -3\sqrt{2}+R] \cup [3\sqrt{2}-R,3\sqrt{2}+R]$, one of the two balls
is completely contained in the half-space $H_A$ or $H_B$ which fixes its center at $A$ or $B$. Taking now 
$\vv= \vect{AB}$, and since the other ball is between $A$ and $B$, we see that the function
${\bf g}^{AB} \cdot  \vect{AB}$ has a constant sign on this ball which prevents the integral to be zero
and make this case impossible.
\end{itemize}
Thus, we are always in the first case: this gives the conclusion. 

\smallskip
\noindent{\bf Computation of the sign of the Jacobian of $G_1$ at its zeros.}
The partial derivatives of the function $G_1$, as function of $A,B$, can be computed explicitly. 

We have the following general formula. Let $h: [0, +\infty) \ra \R_+^*$ be of class $C^1$ and $B_{O,R}$ the ball centred at $O$ of radius $R$. For every $i=1, \dots, N$ we denote
$$f_i^{O,R} (A)= \int_{B_{O,R}} h(d_A(Y)) (y_i-x_i) dy, $$
where $A=(x_i)_i$ and $Y=(y_i)_i$. 
Then, we compute the derivatives at the center of the ball $A=O$, 
$$\frac {\partial f_i^{O,R}}{\partial x_j} \Big  |_{A=O} = \int_{B_{O,R}} h'(d_O(Y)) \frac{(x_j-y_j)(y_i-x_i)}{d_O(Y)} + h(d_O(Y)) (-\delta_{ij})dy.$$
For $i\not =j$, we get $\frac {\partial f_i^{O,R}}{\partial x_j} (O) =0$. For $i=j$, we get
$$\frac {\partial f_i^{O,R}}{\partial x_i} \Big |_{A=O} =-  \int_{B_{O,R}} h'(d_O(Y)) \frac{(y_i-x_i)^2}{d_O(Y)} + h(d_O(Y))  dy$$
$$=-  \int_{B_{O,R}} \frac{\partial }{\partial y_i} [h(d_O(Y)) (y_i-x_i)]  dy= -\int_{\partial B_R} h(d_O(Y)) (y_i-x_i) n_i d \sigma_Y$$
$$= -R h(R) \int _{\partial B_{O,R}}  n_i^2 d \sigma_Y=-\frac{R h(R) P(B(0,R))}{N}\,=-\omega_N R^N h(R).$$
In a similar manner, for a fixed point $A=(x_i^A)$, and for variable points $B=(x_i)_i$, we consider the generic function
$$f_i ^{A,O,R} (B)= \int_{B_{O,R}} h(d_B(Y)) \frac{(x_i-x^A_i)(\vect{AB}\cdot \vect{BY}) }{\|AB\|^2}dy, $$
We assume that $A$ and $O$ are both on the first axis, so that $\vect{AO}= \beta \ee_1$, for some $\beta \in \R^*$. 
Denoting $O_\vps= (x_1^O+\vps, x^O_2, \dots, x^O_N)$, we  have
$$\frac {\partial f_1^{A,O,R}}{\partial x_1} \Big |_{B=O}= \frac{d}{d\vps}\Big |_{\vps =0} \int_{B_{O,R}} h(d_{O_\vps}(Y)) \frac{(\beta + \vps)((\beta+\vps) \ee_1\cdot \vect{O_\vps Y}) }{(\beta +\vps)^2}dy=-\frac RN \int _{\partial B_{O,R}} h(d_O(Y)) d\sigma_Y. $$
For $i \ge 2$, we have $\vect{AO_\vps} = \vec{AO}+ \vps \ee_i$, and plugging in the definition of $f_1$, we get 
$$\frac {\partial f_1^{A,O,R}}{\partial x_i} \Big |_{B=O}= \int_{B_{O,R}} h'(d_O(Y)) \frac{(-y_i)}{d_O(Y)}(y_1-x_1^O) d y + \int_{B_{O,R}} h(d_O(Y))\frac{y_i}{\beta} dy =0.$$
In order to compute $\frac {\partial f_2^{A,O,R}}{\partial x_1} \Big |_{B=O}$ we consider the perturbation $O_\vps= (x_1^O+\vps, x^O_2, \dots, x^O_N)$, and notice that
$$f_2^{A,O,R}(O_\vps)= 0. $$
In order to compute $\frac {\partial f_2^{A,O,R}}{\partial x_2} \Big |_{B=O}$, we consider the perturbation $O_\vps= (x_1^O, x^O_2+\vps, \dots, x^O_N)$ and notice that
$$ f_2^{A,O,R}(O_\vps)= \int_{B_{O,R}} h(d_{O_\vps} (Y)) \frac{\vps(\beta (y_1-x_1^O) + \vps (y_2-\vps))}{\beta^2 +\vps ^2},$$
and
$$\frac {\partial f_2^{A,O,R}}{\partial x_2} \Big |_{B=O}=\int_{B_{O,R}} h'(d_{O} (Y))\frac{-y_2}{d_{O} (Y)} \cdot 0 dy + \int_{B_{O,R}} h(d_{O} (Y))\frac{y_1-x_1^O}{\beta} dy=0.$$
In order to compute $\frac {\partial f_2^{A,O,R}}{\partial x_i} \Big |_{B=O}$, $i\ge 3$ we consider the perturbation $\vect{AO_\vps}= \beta \ee_1+ \vps \ee_i$, and notice that
$$f_2^{A,O,R}(O_\vps)= 0. $$
For the computation of the Jacobian of $G_1$ at the points $(X^\sharp, X^*)$ and $(X^*, X^\sharp)$ we recall that
$$G_1(A,B)= \big (\int_{B^\sharp} {\bf g}^{AB} \cdot \ee_i  dx , \int_{B^*}{\bf g}^{AB} \cdot \ee_i  dx \big ), i= 1, \dots, N.$$
Around the zero $(X^\sharp, X^*)$, the expression of $G_1$ is 
$$ \big (\int_{B^\sharp} {\bf g}_{A} \cdot \ee_i  dx , \int_{B^*}{\bf g}_{B} \cdot \ee_i  dx -  2 {\bf g}_{B} \cdot \vect{AB} \frac{\vect{AB}\cdot \ee_i}{\|AB\|^2}\big ), i= 1, \dots, N,$$
or, in terms of our notations
$$ G_1(A,B)= (f_i^{X^\sharp, R} (A), f_i^{A, X^*,R}(B)),  i= 1, \dots, N,$$
with $h(r)= \frac{G_R(r)}{r}$. 
Then, the Jacobian matrix at $(X^\sharp, X^*)$ is diagonal, with all  elements on the diagonal equal to 
$-\omega_N R^N h(R)$, except the one on position $(N+1,N+1)$ which equals $\omega_N R^N h(R)$. 
Its determinant equals
$$ - \Big ( \omega_N R^N h(R) \Big ) ^{2N},$$
which is a negative number. 

The same value is obtained at the point $(X^*, X^\sharp)$ as the sign of $\beta$ does not influence the value of the derivatives. 

As conclusion, the topological degree of $F$ at $0$ is equal to $2$ which leads to the existence of (at least)
two solutions of $F(A,B)=0$ in $[-M,M]^{2N}\sm V$.

\section{Proof of Theorem \ref{bh03} }\label{s3}

In this section we shall prove Theorem \ref{bh03}. We start with the following observation. In \eqref{bh08}, one can replace $C^\infty_c(\R^N)$ with $W^{1,\infty}(\R^N)$. Indeed, for a function $u \in W^{1,\infty}(\R^N)$ such that $\int_{\R^N} \rho u dx=0$, both terms $\int_{\R^N}\rho u^2 dx$ and $\int_{\R^N}\rho |\nabla u|^2 dx$ are well defined. Moreover, there exists a sequence of functions $\vphi_n \in C^\infty_c(\R^N)$ such that
$$\int_{\R^N}\rho \vphi_n dx =0, \;\; \lim_\nif \int_{\R^N} \rho \Big (|\nabla \vphi_n -\nabla u|^2 + |\vphi_n-u|^2 \Big )dx =0.$$ 
The construction is standard by cut-off and convolution, one has to be careful only to the orthogonality on $\rho$. Let $\vphi \in C^\infty_c(\R^N, [0,1])$ such that $\vphi=1$ on $B_1$ and $\vphi=0$ on $\R^N \sm B_2$. We introduce for every $
\delta >0$, $\vphi_\delta (x):= \vphi(\delta x)$ and the constant $c_\delta$ such that
$$\int_{\R^N} \rho \vphi_\delta (u-c_\delta) dx=0.$$
We observe that 
$$c_\delta \int_{\R^N} \rho \vphi_\delta dx= \int_{\R^N} \rho \vphi_\delta u dx,$$
and for $\delta \ra 0$ we get $c_\delta \ra 0$. This is a consequence of 
$$\lim_{\delta \ra 0} \int_{\R^N} \rho \Big (|\nabla (\vphi_\delta u) -\nabla u|^2 + |\vphi_\delta u-u|^2 \Big )dx =0.$$ 
Now, for fixed $\delta >0$, we consider a convolution kernel $(\xi_\vps)_\vps$ and a constant $c_{\delta, \vps}$ such that 
$$\int_{\R^N} \rho \xi_\vps \ast (\vphi_\delta (u-c_{\delta, \vps})) dx=0.$$
On the one hand 
$$c_{\delta, \vps}  \int_{\R^N} \rho \xi_\vps \ast \vphi_\delta dx = \int_{\R^N} \rho \xi_\vps \ast (\vphi_\delta u)dx,$$
hence for $\vps \ra 0$ we get $c_{\delta, \vps}  \ra c_\delta$. On the other hand
$$\lim_{\vps \ra 0} \int_{\R^N} \rho \Big (|\nabla ( \xi_\vps \ast (\vphi_\delta u)) -\nabla (\vphi_\delta u)|^2 + | \xi_\vps \ast (\vphi_\delta u)-\vphi_\delta u|^2 \Big )dx =0,$$ 
which concludes the proof by a diagonal argument.

\medskip
\noindent {\bf The case $k=1$ (extension of the Szeg\"{o}-Weinberger result).}
Since  inequality \eqref{bh09} that we want to prove is scale invariant, we can assume that $\int_{\R^N} \rho dx =1$. Let $r_1$ be the radius of the ball of volume equal to $1$.

If $\rho$ has bounded support, the proof follows step by step the geometric case. The existence of a point $A$ such that
\begin{equation}\label{bh12}
\forall i=1, \dots, N\;\; \int _{\R^N} \rho  {\bf g}_A(x) \cdot \ee_i dx =0
\end{equation}
is done using the same fixed point argument used by Weinberger (see \cite[Lemma 6.2.2] {He06}). 
The function $G=G_{r_1}$, which enters  in the definition of ${\bf g}_A$ above, is associated to $r_1$. 

Then, the proof follows step by step, the final argument being the displacement of the mass of $\rho$ towards $1_{B_{r_1}}$. 

If $\rho$ has unbounded support, the existence of a point $A$ satisfying \eqref{bh12} can be done by approximation. Note that for every $i$ the function ${\bf g}_A(x) \cdot \ee_i $ belongs to $W^{1, \infty} (\R^N)$. Let $R_n \ra +\infty$ and consider $A_n$ a point satisfying the orthogonality relations  \eqref{bh12} for the density $\rho 1_{B_{R_n}}$ (which has bounded support) and for the ${\bf g}$-functions defined with $r_1$. 
If, for a sub-sequence, $(A_n)_n$ remains bounded, then by compactness we find a limit $A$ such that $A_n \ra A$.  
It can be easily observed that the orthogonality relations \eqref{bh12} pass to the limit, since $\| {\bf g}_{A_n} \cdot \ee_i\|_\infty \le G(r_1)$. Hence $A$ satisfies \eqref{bh12} for $\rho$.

Assume for contradiction that $d_O(A_n)\ra +\infty$. We fix a radius $\ov R$ such that
$$\int _{B_{\ov R}} \rho dx = \frac 23.$$
For $n$ large enough such that $r_n \ge \ov R$, we denote $\vv_n=\frac{1}{\|\vect{A_nO}\|} \vect {A_nO}$. By the choice of $A_n$, we have
$$\int _{B_{R_n}} {\bf g}_{A_n} \cdot \vv_n dx =0,$$
which gets in contradiction with
$$\lim_{n \ra +\infty} \int _{B_{\ov R}}{\bf g}_{A_n} \cdot \vv_n dx =\frac 23 G(r_1), \;\;\mbox { and }\;\; \int _{B_{R_n} \sm B_{\ov R}}|{\bf g}_{A_n} \cdot \vv_n| dx\le \frac 13G(r_1).$$
Hence, $(A_n)_n$ remains bounded and we can build the functions of \eqref{bh12}. The proof ends using a mass displacement  argument,  pushing forward the measure  $\rho dx$ on $1_{B_{r_1}} dx$. 

\medskip
\noindent {\bf The case $k=2$.}
Assume that $\int_{\R^N} \rho dx =1$ and that $\tilde \mu_2(\rho) > \mu_2^*$. Let us denote $r_{\frac 12}$ the radius of the ball of volume $\frac 12$.

There are two difficulties. Along with the fact that the support of $\rho$ may be unbounded, there is a new difficulty: there is no necessarily existence of an {\it  eigenfunction} associated to $\tilde \mu_1(\rho)$, by eigenfunction understanding a function for which the infimum is attained in the definition of $\tilde \mu_1(\rho)$. The orthogonality on the first eigenfunction, both in $L^2$ and $H^1$, was an important point of the proof in the geometric case. Indeed, in the Rayleigh quotient estimating the second eigenfunction, the scalar product $\int_{\R^N} \rho \nabla u_1 \nabla g_i dx$ was not present, being equal to $0$.

Let us fix $\vps >0$ and consider $u_1\in W^{1, \infty}(\R^N)$ such that $\int_{\R^N}\rho u_1 dx = 0$, $\int_{\R^N}\rho u_1^2 dx = 1$ and
\begin{equation}\label{bh14}
\tilde \mu_1(\rho) \le \int_{\R^N}\rho |\nabla u_1|^2 dx < \tilde \mu_1(\rho)+  \vps.
\end{equation}
Let us prove the existence of two points $A \not=B$ (one of them being possibly at infinite distance from the origin) such that 
\begin{equation}\label{bh07.1}
\forall i=1, \cdots, N, \;\; \int_{\R^N} \rho {\bf g}^{AB}\cdot {\ee_i} dx= \int_{\R^N} \rho {\bf g}^{AB}\cdot {\ee_i} u_1 dx=0.
\end{equation}
Above, we abuse of the notation ${\bf g}^{AB}$ even if one of the points $A$ and $B$ is formally at infinite distance from the origin. The exact meaning is given below. 

Let $R_n \ra +\infty$. We apply step by step the method of Section \ref{s2} to the functions
$$1_{B_{R_n}} \rho, \;\; 1_{B_{R_n}} \rho u_1,$$
and find a couple of points $ (A_n,B_n)$ such that
\begin{equation}\label{bh07.2}
\forall i=1, \cdots, N, \;\; \int_{B_{R_n}} \rho {\bf g}^{A_nB_n}\cdot {\ee_i} dx= \int_{B_{R_n}} \rho {\bf g}^{A_nB_n}\cdot {\ee_i} u_1 dx=0.
\end{equation}
If both sequences $(A_n)_n$, $(B_n)_n$ stay bounded, we can assume (up to extracting a sub-sequence) that $A_n \ra A$, $B_n\ra B$. If $A \not = B$, then all equalities in \eqref{bh07.2} pass to the limit to \eqref{bh07.1}. If $A=B$, then taking a further sub-sequence such that
$$\frac{\vect{A_nB_n}}{\|A_nB_n\|} \ra \vv \in S^{N-1}$$
we would get in the limit that 
$$\int_{\R^N} \rho {\bf g}^{AA}\cdot {\vv} dx=0,$$
where $ {\bf g}^{AA}$ is the  pointwise limit of the sequence $ {\bf g}^{A_nB_n}$. This is is not possible since ${\bf g}^{AA}\cdot {\vv} $ is a negative function.

If  $(A_n)_n$ stays bounded, and $d_O(B_n) \ra +\infty$, we can assume that $A_n \ra A$ and obtain that  the limit of ${\bf g}^{A_nB_n}$ equals $ {\bf g}_A:= {\bf g}^{A\infty} $. Then, the functions $({\bf g}_A \cdot \ee_i)_i$ satisfy \eqref{bh07.1}. A similar assertion holds if $B_n \ra B$, $d_O(A_n) \ra +\infty$ and $\frac{\vect{A_nB_n}}{\|A_nB_n\|} \ra \vv \in S^{N-1}$, in which case the limit is described by 
$${\bf g}_B \cdot \ee_i -2 {\bf g}_B \cdot \vv (\vv \cdot \ee_i):= {\bf g}^{\infty B}\cdot \ee_i, i=1, \dots,N $$
satisfy the orthogonality relations \eqref{bh07.1}. 

We prove now that both sequences $(A_n)_n$, $(B_n)_n$ can not go unbounded simultaneously, since  the orthogonality on constants (in relations  \eqref{bh07.2}) would be contradicted. Indeed, denote $O_n$ the projection of $O$ on the line $A_nB_n$. Fix $\ov R$ large enough such that 
$$\int_{B_{\ov R}} \rho dx = \frac 34.$$
We can assume (possibly exchanging the notations and extracting further sub-sequences) that $\|A_nO_n\|\le \|B_nO_n\|$.

If for an infinite number of indices we have 
$$\widehat {O_nOA_n} \le \frac  {\pi}{4}, $$
then taking $\vv_n= \frac{\vect{O_nO} }{\|O_nO\|}$, we get
$$\liminf_\nif \int_{B_{\ov R}} \rho  {\bf g}^{A_nB_n}\cdot {\vv_n} dx \ge \frac{1}{\sqrt{2}} \frac 34 G(r_{\frac 12})> \frac 12 G(r_{\frac 12}),$$
contradicting the hypotheses  \eqref{bh07.2}, as it is not possible that $ \int_{B_{R_n}} \rho {\bf g}^{A_nB_n}\cdot {\vv_n} dx=0$.

If for an infinite number of indices we have 
$$\widehat {O_nOA_n} \geq \frac  {\pi}{4} \Leftrightarrow \widehat {OA_nO_n} \le \frac {\pi}{4}, $$
we take $\vv_n= \frac{\vect{A_nB_n}}{\|A_nB_n\|}$ and arrive to the same conclusion.

We finally conclude with the validity of \eqref{bh07.1}. By abuse of notation, we continue to denote the set of such functions $({\bf g} ^{AB} \cdot {\ee_i})_i$, even though, one of the points is at $\infty$ (i.e. the functions ${\bf g}^{A\infty} \cdot \ee_i= {\bf g}_{A} \cdot \ee_i$ and ${\bf g}^{\infty B}\cdot \ee_i= {\bf g}_B \cdot \ee_i -2 {\bf g}_B \cdot \vv (\vv \cdot \ee_i)$).  Let us denote by 
$${\mathcal G}= \{ (g_i)_{i=1, \dots, N} : \exists A, B, \; g_i = \frac{{\bf g}^{AB} \cdot \ee_i}{(\int_{\R^N} \rho ({\bf g}^{AB} \cdot \ee_i)^2 dx)^{\frac 12}}, \forall i=1,\dots,N, \int_{\R^N}\rho {\bf g}^{AB} \cdot \ee_i dx=0\}.$$
The family ${\mathcal G}$ is not empty, and moreover there exist at least one package of $N$ functions $(g_i)_i$ such that $\int_{\R^N}\rho g_i u_1=0$, as we have proved above.
We have, in particular, 
$$\forall i=1, \dots, N, \;\; \int_{\R^N} \rho g_i^2 dx =1.$$
We observe that the set ${\mathcal G}$ is sequentially compact as if $({\bf g}^{A_nB_n} \cdot \ee_i )_i \in {\mathcal G}$, at least one of the sequences $(A_n)_n$ or $(B_n)_n$ has to stay bounded. 

We split the discussion in two cases.

\smallskip
\noindent {\bf Case 1.} Assume that there exists some function $u_1 \in W^{1, \infty} (\R^N)$ such that $\int_{\R^N}\rho u_1 dx = 0$, $\int_{\R^N}\rho u_1^2 dx = 1$ and
\begin{equation}\label{bh12.1}
\tilde \mu_1(\rho) = \int_{\R^N}\rho |\nabla u_1|^2 dx.
\end{equation} 
Since we get
$$ \int_{\R^N}\rho g_i dx = \int_{\R^N}\rho g_i u_1 dx= \int_{\R^N}\rho \nabla g_i \nabla u_1 dx =0,$$
the proof follows step by step the geometric case, by the mass displacement argument. 

\smallskip
\noindent {\bf Case 2.} Assume that there does not exists a function $u_1 \in W^{1, \infty} (\R^N)$ such that \eqref{bh12.1} holds. In this case, $u_1$ will satisfy only inequality \eqref{bh14}. We introduce the following numbers independent on the choice of $u_1$.  
$$m:= \inf \{ \frac{\int_{\R^N}\rho |\nabla g_k |^2 dx }{\int_{\R^N}\rho |g_k |^2 dx} : k=1, \dots, N, \; (g_i)_i\in {\mathcal G}\}.$$
$$M:= \sup \{ \frac{\int_{\R^N}\rho |\nabla g_k |^2 dx }{\int_{\R^N}\rho |g_k |^2 dx} : k=1, \dots, N, \; (g_i)_i\in {\mathcal G}  \}.$$
The values $m$ and $M$ are attained as a consequence of the same compactness argument described above. Therefore, we have the strict inequality
$$\tilde \mu_1(\rho) < m.$$
We give the following.
\begin{lemma}\label{bh13}
There exists $C>0$ such that $\forall \vps \in (0, \frac {m-\tilde \mu_1(\rho)}{2})$ and for every $(g_i)_i\in {\mathcal G}$ satisfying
$\int_{\R^N} \rho g_i u_1 dx =0$ we have
\begin{equation}\label{bh14.1}
\forall i=1, \dots, N, \;\;\; \tilde \mu_2(\rho)\le \int_{\R^N}\rho |\nabla g_i |^2 dx +C\vps.
\end{equation} 
\end{lemma}
\begin{proof}
Assume the set of functions $(g_i)_i$ satisfies \eqref{bh07.1}. We write, for some $i \in \{1, \dots, N\}$,
$$\forall t \in \R, \;\;\tilde \mu_1(\rho) \le \frac{\int_{\R^N}\rho |\nabla u_1+ t\nabla g_i |^2 dx}{\int_{\R^N}\rho | u_1+t g_i |^2 dx}= \frac{\int_{\R^N}\rho |\nabla u_1+ t\nabla g_i |^2 dx}{1+t^2}.
$$
Direct computations and the knowledge of $\int_{\R^N} \rho |\nabla u_1|^2 dx \le \tilde \mu_1(\rho)+ \vps$, give
$$\forall t \in \R,\;\;\; 0 \le \vps + 2t \int_{\R^N}\rho \nabla u_1\nabla g_i dx + t^2 (\int_{\R^N}\rho |\nabla g_i|^2 dx-\tilde \mu_1(\rho)).$$ 
For 
$$t = -\frac{\int_{\R^N}\rho \nabla u_1\nabla g_i dx}{\int_{\R^N}\rho |\nabla g_i|^2 dx-\tilde \mu_1(\rho)},$$
we get
$$0\le \vps - \frac{(\int_{\R^N}\rho \nabla u_1\nabla g_i dx)^2 }{\int_{\R^N}\rho |\nabla g_i|^2 dx-\tilde \mu_1(\rho)},$$
or
\begin{equation}\label{bh15}
(\int_{\R^N}\rho \nabla u_1\nabla g_i dx)^2 \le \vps (\int_{\R^N}\rho |\nabla g_i|^2 dx-\tilde \mu_1(\rho))\le \vps (M-\tilde \mu_1(\rho))
\end{equation}
where the uniform bound on the gradient of $g_i$ has been obtained at \eqref{boundgrad}.
This inequality gives a control of the scalar product $\int_{\R^N}\rho \nabla u_1\nabla g_i dx$ by $\sqrt{\vps}$.

By definition, we have
$$\tilde \mu_2(\rho)\le \sup_{t \in \R} \frac{\int_{\R^N}\rho |\nabla u_1+ t\nabla g_i |^2 dx}{1+t^2}.$$
For $t \ra \pm \infty$, the right hand side converges to the same value $\int_{\R^N}\rho |\nabla g_i |^2 dx$. If this is the supremum, the lemma is proved. Otherwise, we search the values of $t$ which are critical for the right hand side above. Performing the derivative in $t$, those critical values have to satisfy
$$-t^2 (\int_{\R^N}\rho \nabla u_1\nabla g_i dx) + t (\int_{\R^N}\rho |\nabla g_i|^2 dx- \int_{\R^N}\rho |\nabla u_1|^2 dx)+ \int_{\R^N}\rho \nabla u_1\nabla g_i dx=0.$$
If $\int_{\R^N}\rho \nabla u_1\nabla g_i dx=0$, then the only critical point is $t=0$ and in this case, this corresponds to a minimum for the Rayleigh quotient, the maximum being achieved for $t \ra \pm \infty$. If $\int_{\R^N}\rho \nabla u_1\nabla g_i dx \not=0$, then the two real roots $t_1,t_2$  satisfy

$$t_1t_2=-1, \;\; t_1+t_2= \frac{\int_{\R^N}\rho |\nabla g_i|^2 dx- \int_{\R^N}\rho |\nabla u_1|^2 dx}{ \int_{\R^N}\rho \nabla u_1\nabla g_i dx}.$$
In particular,  the second equality leads to
$$|t_1+t_2 |\ge \frac{m-(\tilde \mu_1(\rho) +\vps)}{|\int_{\R^N}\rho \nabla u_1\nabla g_i dx|}\ge \frac 12 \frac{m-\tilde \mu_1(\rho)}{|\int_{\R^N}\rho \nabla u_1\nabla g_i dx|}\ge  \frac{m-\tilde \mu_1(\rho)}{2\sqrt{\vps(M-\tilde \mu_1(\rho))}} .$$
We conclude that for some constant $C$, independent on $\vps$, we  have (possibly switching the indices)
$$|t_1| \le C\sqrt{\vps}, \;\; |t_2|\ge \frac 1C \sqrt{\vps}.$$
Evaluating the Rayleigh quotient in $t_1, t_2$ and taking into account that $\vps$ is small and $\tilde \mu_2(\rho) \ge 2^\frac 2N |B|^\frac 2N \mu_1(B)- |B|^\frac 2N \mu_1(B)$, we observe that the maximum is attained in $t_2$, which leads to 
$$\tilde \mu_2(\rho) \le  \int_{\R^N}\rho |\nabla g_i |^2 dx + C^2 |B|^\frac 2N \mu_1(B) \vps + 2C \vps \sqrt{\vps(M-\tilde \mu_1(\rho))},$$
 concluding the proof of the lemma.
\end{proof}
Going back to the proof of Theorem \ref{bh03}, we can use inequalities \eqref{bh14.1} as in the geometric case
(see the subsection {\it The use of test functions}), to obtain 

$$\tilde \mu_2(\rho)\le \mu_2^* + C \vps.$$
Making $\vps \ra 0$, the inequality is proved. 

If equality occurs, then the mass displacement should involve only a set of zero measure, otherwise the inequality is strict, independent on $\vps$.

\bigskip\ack The first author is indebted to Edouard Oudet for providing computational support leading to the formulation of Theorem 3. As well, the first author is thankful to Mark Ashbaugh, Carlo Nitsch  and Bozhidar Velichkov for their remarks and very stimulating discussions on this topic. In particular, Carlo Nitsch suggested the use of the Poincaré-Miranda theorem as a very straight and intuitive alternative to the use of the Brouwer theorem in Weinberger's proof. Its application in Theorem 1 would be a very nice alternative which, for the moment, faces the difficulty of handling the unknown function $u_1$ and the interpretation of the test function ${\bf g}^{AB}$ across the diagonal set.

\bigskip
\bibliographystyle{mybst}
\bibliography{References}

\end{document}